\documentclass{amsart}
\usepackage[left=1in,right=1in,top=1.2in,bottom=1in]{geometry}
\usepackage[utf8]{inputenc}
\usepackage{setspace}
\usepackage{amsmath,amssymb}
\usepackage{comment}
\usepackage{xcolor}
\usepackage{url}
\usepackage[colorlinks=true,allcolors=blue,backref=page]{hyperref}
\usepackage[noabbrev,capitalize,nameinlink]{cleveref}
\usepackage{bbm}

\newtheorem{theorem}{Theorem}
\newtheorem{definition}[theorem]{Definition}

\newtheorem{proposition}[theorem]{Proposition}
\newtheorem{lemma}[theorem]{Lemma}

\newtheorem{corollary}[theorem]{Corollary}

\newtheorem{fact}[theorem]{Fact}

\newtheorem*{claim*}{Claim}

\theoremstyle{remark}
\newtheorem*{remark*}{Remark}

     \def\l{\ell} \def\la{\lambda}

\newcommand{\eps}{\varepsilon}
\renewcommand{\epsilon}{\eps}
\renewcommand{\P}{\mathbb{P}}
\newcommand{\one}{\mathbbm{1}}

\newcommand{\precdot}{\prec\mathrel{\mkern-5mu}\mathrel{\cdot}}
\newcommand{\cB}{\mathcal{B}}
\newcommand{\tlambda}{\widetilde{\lambda}}
\newcommand{\tmu}{\widetilde{\mu}}
\newcommand{\tnu}{\widetilde{\nu}}

\begin{document}
	
	\title[Realizability of hypergraphs and contingency tables]{Realizability of hypergraphs and high-dimensional contingency tables with random degrees and marginals}
	\author{Nicholas Christo}

	\address{University of Illinois, Chicago. Dept of Mathematics, Statistics and Computer science.}
	\email{nchrist5@uic.edu}
	
	\author{Marcus Michelen}
	\address{University of Illinois, Chicago. Dept of Mathematics, Statistics and Computer science.}
	\email{michelen@uic.edu}
	
	\begin{abstract}
		A result of Deza, Levin, Meesum, and Onn shows that the problem of deciding if a given sequence is the degree sequence of a 3-uniform hypergraph is NP complete.  We tackle this problem in the random case and show that a random integer partition can be realized as the degree sequence of a $3$-uniform hypergraph with high probability.  These results are in stark contrast with the case of  graphs, where a classical result of Erd\H{o}s and Gallai provides an efficient algorithm for checking if a sequence is a degree sequence of a graph and a result of Pittel shows that with high probability a random partition is not the degree sequence of a graph.  
		
		By the same method, we address analogous realizability problems about high-dimensional binary contingency tables.  We prove that if $(\lambda,\mu,\nu)$ are three independent random partitions then with high probability one can construct a three-dimensional binary contingency table with marginals $(\lambda,\mu,\nu)$.  Conversely, if one insists that the contingency table forms a pyramid shape, then we show that with high probability one cannot construct such a contingency table.  These two results confirm two conjectures of Pak and Panova.
	\end{abstract}

	\maketitle

	\section{Introduction}

	When can one realize a sequence of integers as the degree sequence of some graph? This combinatorial decision problem is fully answered thanks to a classical theorem of Erd\H{o}s and Gallai \cite{ErdosGallai1960} which provides an easy-to-check combinatorial condition on the would-be degree sequence.  A closely-related theorem of Gale and Ryser \cite{Gale1987, Ryser_1957} provides a similar criterion to decide if a bipartite graph exists where the degree sequences of each side of the bipartition are prescribed.  By using the Erd\H{o}s-Gallai criterion along with proving new results on the structure of random integer partitions, a 1999 theorem of Pittel \cite{PITTEL1999123} shows that given a random integer partition, with high probability it \emph{is not} the degree sequence of some graph, thereby confirming a 1982 conjecture of Wilf.

	If one asks similar questions about hypergraphs, then the story becomes much richer and more complex than the graph setting.    Prior to Erd\H{o}s-Gallai, Havel \cite{havel1955remark} and Hakimi \cite{hakimi1962realizability} independently described polynomial-time algorithms for the graph realization problem given a proposed degree sequence. Dewdney \cite{dewdney1975degree} generalized the Havel-Hakimi algorithm to hypergraphs, however it is not efficient.  Deza, Levin, Meesum, and Onn \cite{deza2018optimization} showed that deciding if a given sequence is the degree sequence of a 3-uniform hypergraph is NP-Complete, thereby resolving a 30-year old conjecture of Colbourn, Kocay and Stinson \cite[Prob. 3.1]{colbourn86}; this is in sharp contrast to the graph setting, in which case the Erd\H{o}s-Gallai and Havel-Hakimi theorems provide easy-to-implement polynomial-time algorithms.  We note that there are some results known for the hypergraph realization problem under certain cases (see, e.g.  \cite{frosini2013degree,behrens2013new,aldosari2019enumerating,li2023dense} and the references therein for more).  
	
	Our first main theorem shows that in contrast to the graph case---and despite these hardness results---a random partition in fact \emph{is} the degree sequence of a hypergraph with high probability.  
	
	\begin{theorem}\label{th:hypergraph}
		There is a constant $c > 0$ so that if $\lambda$ is an integer partition of $3n$ chosen uniformly at random, then with probability at least $1 - e^{-c\sqrt{n}}$ there is a $3$-uniform hypergraph with degree sequence given by $\lambda$.
	\end{theorem}
	
	We note that since there are $\exp(\Theta(\sqrt{n}))$ many integer partitions of $3n$ and there exist integer partitions that are not the degree sequence of a $3$-uniform hypergraph---such as the partition of $3n$ consisting of only one part---this theorem is sharp up to the value of $c > 0$.  Our proof of \cref{th:hypergraph} is in fact algorithmic, and while we do not explicitly analyze the run-time in this work, it can be shown to run in polynomial time.  A natural question is whether one can require that the hypergraph is linear; we show in \cref{sec:linear-hypergraph} that with high probability there is no such $3$-uniform linear hypergraph with degree sequence $\lambda$.

	We prove  \cref{th:hypergraph} by working in a more general setting of contingency tables.  In the case of graphs, if one identifies a bipartite graph with its biadjacency matrix, then the question of bipartite graph realizability becomes a question of asking if one can construct a zero-one matrix with prescribed row and column sums.  Integer arrays with prescribed row and column sums are often referred to \emph{contingency tables}; if the entries are restricted to be only zeros and ones, then such arrays are called \emph{binary contingency tables}.  In this language, the Gale-Ryser theorem provides a criterion for the existence of a binary contingency table with prescribed row and column sums.

	If one increases from two-dimensional arrays to three-dimensional arrays and asks an analogous realizability problem, then no such criterion is known.  More specifically, a \emph{three-dimensional binary contingency table} is a three-dimensional array $(M(i,j,k))_{i,j,k \geq 1}$ with $M(i,j,k) \in \{0,1\}$ for all $i,j,k$.  Its \emph{marginals} are the vectors $\lambda,\mu,\nu$ defined via 
	$$\lambda_i = \sum_{j,k}M(i,j,k)\qquad \mu_j = \sum_{i,k} M(i,j,k)\qquad \nu_k = \sum_{i,j} M(i,j,k)\,.$$
	Rather than asking for the existence of a certain type of graph with a prescribed degree sequence, this amounts to asking for a tripartite hypergraph, i.e.\ a three-uniform hypergraph whose vertices are split into three sets, where each edge must take exactly one vertex from each of the three sets and degrees within each of the three sets are prescribed by the marginals $\lambda,\mu$ and $\nu$ respectively.  As in the case of \cref{th:hypergraph}, we resolve the realizability problem for three-dimensional binary contingency tables in the random case.  
	
	\begin{theorem}\label{th:main}
		There is a constant $c > 0$ so that if $\lambda,\mu,\nu$ are integer partitions of $n$ chosen independently and uniformly at random then one can construct a three-dimensional binary contingency table with marginals $(\lambda,\mu,\nu)$ with probability at least $1 - e^{-c\sqrt{n}}$.
	\end{theorem}
	As in the case of \cref{th:hypergraph}, this theorem is sharp up to the value of $c > 0$ and is algorithmic with polynomial running time.
	
	\cref{th:main} resolves a conjecture of Pak and Panova \cite{PakP} from a work connecting contingency tables to the \emph{Kronecker coefficients}.   Kronecker coefficients are central objects of study of algebraic combinatorics and representation theory of the symmetric group.  While they are non-negative integers, there is no known combinatorial or geometric description of the Kronecker coefficients.  Further, the problem of deciding whether a Kronecker coefficient is non-zero is an NP-hard problem \cite{ikenmeyer2017}.  We refer the reader to a survey by Pak, Panova, and Vallejo \cite{pak2015combinatorics} for more information on the Kronecker coefficients.   
	
	Pak and Panova \cite{PakP} proved upper bounds on Kronecker coefficients in terms of number of three-dimensional binary contingency tables along with a lower bound in terms of the number of such tables that form a pyramid shape.  More formally, we say that a three-dimensional binary contingency table $M$ is a \emph{pyramid} if for each $i_1 \leq i_2, j_1 \leq j_2, k_1 \leq k_2$ we have $$M(i_1,j_1,k_1) \geq M(i_2,j_2,k_2)\,.$$
	
	Let $T(\la,\mu,\nu)$ denote the number of three-dimensional binary contingency tables with marginals $\la,\mu,\nu$ and $\mathrm{Pyr}(\la,\mu,\nu)$ denote the number of pyramids with marginals $\la,\mu,\nu$.  Pak and Panova showed that the Kronecker coefficient $g(\la,\mu,\nu)$ satisfies \begin{equation}\label{eq:kronecker-sandwich}
		\mathrm{Pyr}(\la,\mu,\nu)\leq g(\la,\mu,\nu)\leq T(\la,\mu,\nu)\,.
	\end{equation}  
	Further, Pak and Panova conjectured \cite[Conj. 8.1]{PakP} that if $(\lambda,\mu,\nu)$ are chosen uniformly at random then $T(\la,\mu,\nu) \geq 1$ and $\mathrm{Pyr}(\la,\mu,\nu) = 0$ with high probability.  \cref{th:main} resolves the first of these conjectures; we also confirm their conjecture on realizability of pyramids:
	
	\begin{theorem} \label{prop:pyramid-limit}
		Let $\lambda,\mu,\nu$ be integer partitions of $n$ chosen independently and uniformly at random.  Then there is a constant $C > 0$ so that $$\P(\mathrm{Pyr}(\lambda,\mu,\nu) > 0) \leq C n^{-0.003}.$$    
	\end{theorem}
	
	Pak and Panova conjecture \cite[Conj. 8.2]{PakP} that if $(\lambda,\mu,\nu)$ are chosen uniformly at random then $g(\lambda,\mu,\nu) > 0$ with high probability;  \cref{th:main} and \cref{prop:pyramid-limit} show that the bounds in \eqref{eq:kronecker-sandwich} are not strong enough to confirm or refute this conjecture.
	
	\subsection{A deterministic criterion}
	
	Both of \cref{th:hypergraph} and \cref{th:main} will be special cases of a deterministic criterion we prove for the realizability of three-dimensional binary contingency tables.  Intuitively, one obstacle to constructing a binary contingency table is to have many large parts and few small parts.  As an extreme example, one can easily see that if we take $\lambda = \mu = \nu = (n)$ then there is no binary contingency table with marginals $(\lambda,\mu,\nu)$.  
	
	We think of integer partitions $\lambda = (\lambda_1,\lambda_2,\ldots)$ as non-increasing sequences of non-negative integers and write $\lambda \vdash n$ if $\sum_j \lambda_j = n$.  We define the set $\cB$ to be the set of triples $(\lambda,\mu,\nu)$ for which there exists a three-dimensional binary contingency table with marginals $(\lambda,\mu,\nu).$

	\begin{definition}
		We say that $\lambda \vdash n$ satisfies \eqref{eq:shape-assumptions} with parameters $(\theta,A,B)$ if 
		\begin{align}
			\sum_{j} \lambda_j \one_{\lambda_j > B \sqrt{n}} \leq \frac{n}{2 A^2}  \qquad \text{ and } \qquad \label{eq:shape-assumptions}
			\sum_{j} \lambda_j \one_{\lambda_j \leq  \sqrt{n}/A} \geq  \theta \frac{n}{A} \,. 
		\end{align}
	\end{definition}
	
	Broadly, this definition asserts that not too much of the partition's ``mass'' comes from large parts and that much of its mass comes from relatively small parts.  Our main technical theorem is that if all three partitions satisfy \eqref{eq:shape-assumptions} then we can construct a $3$-dimensional binary contingency table. 
	
	\begin{theorem}\label{th:deterministic}
		For each $\theta \in (0,1)$ there is an $A_0= A_0(\theta)$ so that the following holds.  For all $B >0$ and $A \geq A_0$ there is an $n_0(\theta,A,B) < \infty$ so that if $n \geq n_0$ and $\lambda,\mu,\nu \vdash n$  all satisfy \eqref{eq:shape-assumptions} with parameters $(\theta,A,B)$ then $(\lambda,\mu,\nu) \in \cB$.
	\end{theorem}
	
	We will combine  \cref{th:deterministic} with a sufficiently strong shape theorem (see \cref{th:shape} and \cref{cor:limit-shape}) about random partitions in order to deduce both  \cref{th:hypergraph} and \cref{th:main} from \cref{th:deterministic}.

	\subsection{Additional related work}
	
	The starting points for deterministic criterion for the graph realization problem are the foundational algorithms of Havel-Hakimi \cite{havel1955remark,hakimi1962realizability} and of Erd\H{o}s-Gallai \cite{ErdosGallai1960}.  In the case of the bipartite graph realization problem---or equivalently the problem of realizability for a binary (two-dimensional) contingency table with prescribed marginals---the Gale-Ryser theorem provides an easy-to-check condition.  In particular, there exists a two-dimensional contigency table with marginals $\lambda$ and $\mu$ if and only if the transpose of $\lambda$ dominates $\mu$ (see  \cref{sec:prelims} for more on the dominance order).  This condition bears quite a bit of similarity to the Erd\H{o}s-Gallai condition.  Using these two similar criteria, Pittel proved in 1999 \cite{PITTEL1999123} that with high probability a random partition is not the degree sequence of a graph and that given two random partitions there is no (two-dimensional) binary contingency table with those margins.  This simultaneously resolved conjectures of Wilf and Macdonald, respectively.
	
	In his first work on the problem, Pittel did not prove a quantitative rate of convergence for this probability.  In a later work, Pittel \cite{pittel2018} proved a quantitative rate of decay that decreased slower than any polynomial.  Conversely, Erd\H{o}s and Richmond proved in 1993 \cite{erdos1993} that a random partition of $2n$ is the degree sequence of a graph with probability at least $\Omega(n^{-1/2})$.  A polynomial upper bound of $n^{-0.003}$ on this probability was subsequently shown by Melczer, Michelen and Mukherjee \cite{MMM}.  A recent work by Balister, Donderwinkel, Groenland, Johnston, and Scott \cite{balister2023counting} identifies the probability that a random sequence is the degree sequence of a graph where here the notion of ``random sequence'' is of the form $n-1 \geq d_1 \geq d_2 \geq \cdots \geq d_n \geq 0$; we note that this has a different distribution than allowing $(d_j)_{j}$ to be a random partition.

	\subsection{Structure of the proof}
	
	Our starting point for \cref{th:deterministic} is to show a basic monotonicity property: if $(\lambda,\mu,\nu) \in \cB$---i.e.\ if there is a binary contingency table with marginals $(\lambda,\mu,\nu)$---and one decreases any of the partitions in the dominance order to obtain $(\tlambda,\tmu,\tnu)$, then $(\tlambda,\tmu,\tnu) \in \cB$.  This is shown in  \cref{sec:prelims}. This will allow us to combine parts of our partition and move pieces to larger parts in order to make the partitions look more alike each other.  
	
	For most of the proof we will remove pieces of our partitions in chunks either by creating cubes or a large square in our binary contingency table.  If we create a cube of side length $a$ for some integer $a$, then for each of the three marginals this amounts to removing $a$ parts of size $a^2$.  If we create a square of side length $a$, then one partition will remove one part of size $a^2$ while the others will remove $a$  parts of size $a$.
	
	To start, we will take away from the largest parts of each partition, namely those of size larger than $B\sqrt{n}$ for some large constant $B$.  This will be done by combining all parts larger than $B \sqrt{n}$ into a single part that is equal to a perfect square and subsequently placing a square in the contingency table.  Under the assumptions of  \cref{th:deterministic}, there will not be too much mass larger than $B \sqrt{n}$ and there will be sufficient mass consisting of small parts, which will allow us to place the desired square in the contingency table.  This step is done in  \cref{sec:large-parts}.
	
	Our next step is to remove the majority of the mass of our partitions.  This is done by iteratively placing cubes in our contingency table while keeping enough small parts to make sure the last stage is still feasible.  This is accomplished in  \cref{sec:no-large}.  
	
	The last step begins once the partition has only small parts i.e.\ once it can be dominated by a rectangle of sufficiently small aspect ratio.  To handle this stage, we may assume that each partition is a rectangle due to the monotonicity of realizability with respect to the dominance order.  Intuitively, in this stage we wish to induct on the size of the partition.  For rectangles, this amounts to decreasing the size of the largest part.  We perform this step by iteratively moving portions of the partition from smaller parts to larger parts---thus increasing the dominance order---and placing cubes into our contingency table.  This has the effect of reducing the size of many of the parts of our partition.  We repeat this step until all remaining parts have shrunk and show that the resulting partition has an even smaller aspect ratio than we began with, thus completing our inductive step.  This is accomplished in \cref{sec:wide}.
	
	We note that we prove  \cref{th:deterministic} in reverse order.  The reason for this is that our approach can be understood by building up a broader description of what set of triples lie in the set $\cB$: first partitions with only small parts, then partitions with no exceptionally large parts, and then partitions satisfying  \cref{th:deterministic}.   We deduce  \cref{th:main} from  \cref{th:deterministic} using a sufficiently strong version of the shape theorem in  \cref{sec:shape-thm}.  
	
	\cref{th:hypergraph} is deduced from  \cref{th:deterministic} (along with the shape theorem) in  \cref{sec:hypergraphs} after proving a some small preliminary statements on hypergraph realizability such as an analogous statement about monotonicity with respect to the dominance order.  In particular, to show the existence of a hypergraph with a given degree seqeuence, we in fact will construct a tripartite hypergraph, i.e.\ a three uniform hypergraph whose vertex set is partitioned into three sets so that each hyperedge contains exactly one vertex from each set.  This is because a three-dimensional binary contingency table can be understood as an adjacency matrix of a tripartite graph.
	
	Finally, we prove  \cref{prop:pyramid-limit} in  \cref{sec:pyramid} by proving that in a pyramid one must have that the transpose of one partition is dominated by the other two partitions.  We then lean on previous work \cite{PITTEL1999123,MMM} to show an upper bound on this probability.

	\section{Preliminaries: building blocks for contingency tables} \label{sec:prelims}
	
	An \emph{integer partition} $\lambda = (\lambda_1,\lambda_2,\ldots)$ is a non-increasing sequence of non-negative integers with finite sum.  We write $\lambda \vdash n$ if $\sum_{j} \lambda_j = n$.  We write $|\lambda| = n$ and $\ell(\lambda) = r$ if $\lambda_r > 0$ and $\lambda_{r+1} = 0$.   We write $\lambda'$ for the \emph{conjugate} of $\lambda$.

	Define $$\mathcal{M} = \left\{ M: \mathbb{N}^3 \to \{0,1\} : \sum_{i,j,k} M(i,j,k)  < \infty \right\}  $$
	to be the set of three-dimensional binary contingency tables.  Define $$\cB = \left\{(\lambda,\mu,\nu) : \exists~M \in \mathcal{M} \text{ with marginals }(\lambda,\mu,\nu)\right\}\,.$$
	
	Throughout the proof, we will often be able to assume that $\lambda = \mu = \nu$, and so we write $\lambda \in \cB$ if $(\lambda,\lambda,\lambda) \in \cB$. 
	Given two integer partitions $\lambda, \mu \vdash n$ recall that $\lambda$ \emph{is dominated} by $\mu$ denoted $\lambda \preceq \mu$ if for all $j$ we have 
	\begin{equation}
		\sum_{i = 1}^j \lambda_i \leq \sum_{i = 1}^j \mu_j\,.
	\end{equation}

	We write $\lambda \prec \mu$ if $\lambda \preceq \mu$ and $\lambda \neq \mu$.  In the partially ordered set of partitions of $n$ under dominance, $\mu$ \emph{covers} $\lambda$---denoted $\lambda \precdot \mu$---if $\lambda \prec \mu$ and there is no $\nu$ so that $\lambda \prec \nu \prec \mu$.  In the dominance order, covering amounts to moving one block of the Ferrers diagram from $\lambda$ to one part just larger than it.  
	In particular, we have the following classical fact (see, e.g., \cite[Prop. 2.3]{BRYLAWSKI1973201}).  
	
	\begin{fact}\label{lem:prec}
		For partitions $\lambda,\mu \vdash n$ we have $\lambda \precdot \mu$ if and only if there is some $j$ so that 
		\begin{equation}\label{eq:precdot-def}
			\mu_{j} = \lambda_{j} + 1\,, \qquad \mu_{j+1} = \lambda_{j+1} - 1\,,\qquad \mu_a = \lambda_a \quad \text{for all}\; a \notin \{j,j+1\}\,.
		\end{equation}
	\end{fact}

	One of our main tools will be that decreasing a partition in the dominance order can only make it easier to construct a contingency table.  
	
	\begin{lemma} \label{lem:dominance-to-cont}
		For $j \in \{1,2\}$ let  $\lambda^{(j)},\mu^{(j)},\nu^{(j)} \vdash n$ and suppose $\lambda^{(1)} \preceq \lambda^{(2)}, \mu^{(1)} \preceq \mu^{(2)}$ and $\nu^{(1)} \preceq \nu^{(2)}$.  If $(\lambda^{(2)},\mu^{(2)},\nu^{(2)}) \in \mathcal{B}$ then $(\lambda^{(1)},\mu^{(1)},\nu^{(1)})\in\mathcal{B}$.
	\end{lemma}
	\begin{proof}
		It is sufficient to assume that $\mu=\mu^{(1)}=\mu^{(2)}$, $\nu=\nu^{(1)}=\nu^{(2)}$, and $\la^{(1)}\preceq \la^{(2)}$, i.e., only the partitions $\la^{(1)}$ and $\la^{(2)}$ differ. 
		We may also assume without loss of generality that $\la^{(1)}\precdot\la^{(2)}$. 
		
		Let $M$ be a contingency table with marginals $(\la^{(2)},\mu, \nu)$. Since $\la^{(1)} \precdot \la^{(2)}$, \cref{lem:prec} shows there is an $i_0$ such that $$\la^{(1)}_{i_0}+1=\la^{(2)}_{i_0}, \;\; \la^{(1)}_{i_0+1}-1=\la^{(2)}_{i_0+1},\; \mathrm{and}\; \la_i^{(1)}=\la_i^{(2)}, \; \forall i\notin\{i_0,i_0+1\}. $$ 
		
		Since the marginals of $M$ are $(\la^{(2)},\mu, \nu)$, we have that $$\lambda^{(2)}_{i_0} = \sum_{j,k}M(i_0,j,k), \qquad \lambda^{(2)}_{i_0+1}= \sum_{j,k} M(i_0+1,j,k)\,.$$
		Note that $\lambda^{(2)}_{i_0} = 1 + \lambda_{i_0}^{(1)} \geq 1 + \lambda_{{i_0}+1}^{(1)} = 2 + \lambda_{{i_0}+1}^{(2)}$ and so in particular we have $\lambda^{(2)}_{i_0} > \lambda^{(2)}_{i_0+1}$.
		
		Thus we must have some pair of indices $j_0,k_0$ for which we have $M(i_0,j_0,k_0) =1$ and \\ 
		$M(i_0+1,j_0,k_0)=0$.  Define the table $M'$ by $$M'(i_0,j_0,k_0) = 0, \qquad M'(i_0+1,j_0,k_0) = 1$$ 
		$$\text{and}\quad M'(i,j,k) = M(i,j,k) \text{ for all }(i,j,k) \notin \{(i_0,j_0,k_0), (i_0+1,j_0,k_0)\}\,.  $$
		
		We then may compute the marginals of $M'$ to be $$ \sum_{i,k} M'(i,j,k) = \sum_{i,k} M(i,j,k) = \mu_j\,,\qquad \sum_{i,j} M'(i,j,k) = \sum_{i,j} M(i,j,k) = \nu_k $$
		$$\sum_{j,k} M'(i,j,k) = \sum_{j,k} M(i,j,k) = \lambda^{(2)}_i = \lambda^{(1)}_i \quad \text{ for }i \notin \{i_0,i_0+1\}    $$
		$$\sum_{j,k} M'(i_0,j,k) = \sum_{j,k} M(i_0,j,k) - 1 = \lambda_{i_0}^{(2)}-1 = \lambda_{i_0}^{(1)}$$ 
		$$\sum_{j,k} M'(i_0+1,j,k) = \sum_{j,k} M(i_0+1,j,k) + 1 = \lambda_{i_0+1}^{(2)}-1= \lambda_{i_0+1}^{(1)}\,.$$
		Hence, $(\la^{(1)},\mu,\nu)\in\mathcal{B}$ as desired.
	\end{proof}

	A useful tool will be that if we can split three partitions into two triples that each lie in $\cB$ then the original triple lies in $\cB$ as well.

	\begin{lemma} \label{lem:split}
		Let $\lambda, \mu, \nu \vdash n$.  Suppose there is some $r \geq 1$ so that we have \begin{align*}(\la_-,\mu_-,\nu_-):= \left((\lambda_1,\ldots,\lambda_r), (\mu_1,\ldots,\mu_r), (\nu_1,\ldots,\nu_r) \right) &\in \cB \\
			\text{ and }\quad (\la_+,\mu_+,\nu_+):=((\lambda_{r+1},\lambda_{r+2},\ldots), (\mu_{r+1},\mu_{r+2},\ldots), (\nu_{r+1},\nu_{r+2},\ldots)) &\in \cB\,.
		\end{align*}   Then $(\lambda,\mu,\nu) \in \cB$.
	\end{lemma}
	\begin{proof}
		We ``glue" together the two associate contingency tables at their corners. Let $M_-$ and $M_+$ be binary contingency tables with marginals $(\lambda_-,\mu_-,\nu_{-})$ and $(\lambda_+,\mu_+,\nu_+)$.  Then define the binary contingency table $M$ via $$M(i,j,k) = \begin{cases}
			M_-(i,j,k) & \text{ if } i,j,k \leq r \\
			M_+(i-r,j-r,k-r) & \text{ if }i,j,k \geq r+1 \\
			0 & \text{ otherwise }
		\end{cases}\,.$$
		We compute $$\sum_{j,k} M(i,j,k) = \sum_{j \leq r,k \leq r} M_-(i,j,k) = \lambda_i \quad \text{ if } i \leq r$$
		$$ \sum_{j,k} M(i,j,k) = \sum_{j \geq r+1,k \geq r+1} M_+(i-r,j-r,k-r) = \lambda_i \quad \text{ if } i \geq r+1\,.$$
		Computing the analogous marginals for each $j$ and $k$ completes the proof.
	\end{proof}

	Our most basic tool is to construct cubes of various sizes.
	
	\begin{fact}    \label{fact:cube}
		For an integer $m \geq 1$, let $\lambda$ be the partition with $m$ parts so that each part has size $m^2$.  Then $\lambda \in \cB$.
	\end{fact}
	\begin{proof}
		The contingency table with $M_{i,j,k} = \delta_{\{i,j,k \in [m]\}}$ has marginals $(\lambda,\lambda,\lambda)$.
	\end{proof}

	\section{Handling partitions with only small parts}
	\label{sec:wide}

	The main result of this section is to show that we can handle partitions whose largest part is fairly small.
	\begin{proposition}\label{prop:partition-wide}
		There is a universal constant $A_0 > 0$ so that the following holds.  Suppose $\lambda,\mu,\nu$ are partitions of $n$ where each has largest part at most $\sqrt{n}/A_0$.  Then $(\lambda,\mu,\nu) \in \mathcal{B}$.
	\end{proposition}
	We will ultimately take $A_0 = 2^{24}$, but this specific value is not particularly important.

	Throughout our proof of  \cref{prop:partition-wide}, we will maintain that all three partitions are equal, and so need only refer to a single partition throughout; this will be possible via \cref{lem:dominance-to-cont} by assuming each of $\lambda,\mu,\nu$ is the maximal partition of $n$ in the dominance order that satisfies $\lambda_1 \leq \sqrt{n}/A_0$.  Our strategy will be to construct many contingency tables that will remove most of the parts of $\lambda$ of size $k$ and decrease them to size $\lfloor k / 2 \rfloor$.  This will handle the vast majority of the parts of size $k$ (provided $k$ is large enough) and so we will just need to decrease the remaining parts of size $k$ in small chunks.

	\subsection{Decreasing most parts of size $k$}  
	
	The goal of this subsection is to reduce most parts of size $k$ to size $\lfloor k / 2 \rfloor$.  
	To set up this step, we say that a partition satisfies \eqref{eq:part-ratio-assumption} if it has \begin{align}
		\text{at most one part with size in } (\lfloor k /2 \rfloor, k) \nonumber \\
		\text{ at most one part with size } < \lfloor k/2 \rfloor, \label{eq:part-ratio-assumption} \tag{$*$}  \\ 
		\text{ all other parts of size }k \text{ or }\lfloor k/2 \rfloor\,.  \nonumber 
	\end{align}

	Our first main step towards  \cref{prop:partition-wide} is to remove most parts of size $k$.
	
	\begin{lemma} \label{lem:wide-first-step}
		Let $\lambda \vdash n$ satisfy $\lambda_1 \leq k \leq \sqrt{n}/A_0$ for some $k \geq 2$.  Then there is a partition $\tlambda$ satisfying \eqref{eq:part-ratio-assumption} with at least $ \frac{8}{9}A_0k- 90\sqrt{k}$ parts of size $\lfloor k/2\rfloor$ and at most $100 \sqrt{k} $ parts  of size $k$; finally, we have that if $\tlambda \in \cB$ then $\lambda \in \cB$.
	\end{lemma}

	To prove \cref{lem:wide-first-step} we define an operator $T_0$ on partitions satisfying \eqref{eq:part-ratio-assumption}.  Let $\lambda$ satisfy \eqref{eq:part-ratio-assumption} and set $m$ to be the smallest integer satisfying $m^2 > 9k$.  Note that $m \leq  3\sqrt{k}+1$ 
	and so we have \begin{equation} \label{eq:m-bounds}
		9k \leq m^2 \leq 9k + 6\sqrt{k}+1\,.
	\end{equation}

	We will attempt to make the first $m$ parts of $\lambda$ to be of size $m^2$ by reducing later parts of size $k$ to size $\lfloor k/2\rfloor$ and moving their mass to the the first $m$ parts.   Repeat this process until the first $m$ parts are of size $m^2$.  After this, if there is more than one part with size in $(\lfloor k/2 \rfloor , k)$, then move pieces from the smaller piece to the larger piece until one has either size $k$ or $\lfloor k/2\rfloor$; once again by \cref{lem:prec} this may only increase the partition in the dominance order.

	If such an operation can be performed, then by  \cref{lem:prec} we note that this yields a new partition $\mu$ with  $\lambda \preceq \mu$ and so by  \cref{lem:dominance-to-cont} it will be sufficient to show $\mu \in \mathcal{B}$.  We may remove the first $m$ parts of $\mu$ using  \cref{lem:split} and  \cref{fact:cube} to yield a new partition ${\tlambda}$.  Let $T_0$ be the operator that takes $T_0(\lambda) = \tlambda$.  We first describe the domain of $T_0$:

	\begin{lemma}\label{lem:T-domain}
		If $\lambda$ satisfies \eqref{eq:part-ratio-assumption} then $T_0$ may be applied provided $\lambda$ has at least $100\sqrt{k}$ many parts of size $k$.  
	\end{lemma}
	\begin{proof}
		We need to reserve the first $m$ parts which will be increased to size $m^2$.  We also need enough parts of size $k$ to fill the first $m$ parts to be of size $m^2$.  The total amount of mass needed is thus $$(m^2 - k)\cdot m \leq (8k + 6\sqrt{k} + 1)(\sqrt{9k + 6\sqrt{k} + 1}) \leq 48 k^{3/2}$$
		for $k \geq 2$.  
		Reducing a part from size $k$ to $\lfloor k / 2 \rfloor$ provides at least mass $k/2$.  This implies that it is enough if we have at least $m + 96\sqrt{k}  \leq 100 \sqrt{k}$ parts of size $k$.
	\end{proof}

	The idea is that while each application of $T_0$ deletes $m$ parts of size $k$, it converts at least $8m$ parts of size $k$ into parts of size $\lfloor k/2 \rfloor$.
	
	\begin{lemma}\label{lem:T-decrease}
		Each application of $T_0$ removes $m$ parts and converts at least $8m$ parts of size $k$ into parts of size $\lfloor k/2 \rfloor$.
	\end{lemma}
	\begin{proof}
		The only parts removed are the first $m$ parts.  Each of the first $m$ parts must increase from $k$ to $m^2$ requiring a total of $m^2 - k \geq 8k$ mass.  Each part lowered from $k$ to $\lfloor k/2\rfloor$ gives at most mass $k$, thus showing the desired lower bound.
	\end{proof}

	We now summarize the desired properties of the operator $T_0$:
	\begin{lemma}\label{lem:T-props}
		The operation $T_0$ above takes a partition $\lambda$ satisfying \eqref{eq:part-ratio-assumption} and outputs a partition $T_0(\lambda) = \tlambda$ satisfying \eqref{eq:part-ratio-assumption} and if $\tlambda \in \cB$ then $\lambda \in \cB$.
	\end{lemma}
	\begin{proof}
		By construction, the only parts altered are of size $k$ and so the only part that may be left over after removing the first $m$ parts is of size in $(\lfloor k/2 \rfloor, k)$; if more than two of these parts existed then they were combined so that only one may exist.  This shows that $\tlambda$ satisfies \eqref{eq:part-ratio-assumption}.
		Finally, at each step we have either increased the partition in the dominance order or removed a partition in $\cB$ and so  \cref{lem:dominance-to-cont}, \cref{lem:split} and \cref{fact:cube} show the last assertion.        
	\end{proof}

	\begin{proof}[Proof of  \cref{lem:wide-first-step}]
		By  \cref{lem:dominance-to-cont}, we may assume without loss of generality that all but at most one part of $\lambda$ are of size $k$. Note that the number of parts of size $k$ in $\lambda$ is at least $A_0k - 1$ since $k \leq \sqrt{n}/A_0$.
		
		Iteratively apply the operator $T_0$ until we are unable to apply $T_0$ to obtain a partition $\tlambda$.  By  \cref{lem:T-decrease}, each application deletes $m$ parts of size $k$ and converts at least $8m$ parts of size $k$ into parts of size $\lfloor k/2 \rfloor$.  By  \cref{lem:T-domain}, we are left with at most $100\sqrt{k}$ parts of size $k$.  This implies that we are left with a partition $\tlambda$ with at least $\frac{8}{9}(A_0k - 1 - 100\sqrt{k}) \geq \frac{8}{9}A_0k - 90\sqrt{k}$ parts of size $\lfloor k/2\rfloor$.  By  \cref{lem:T-props}, the partition $\tlambda$ satisfies \eqref{eq:part-ratio-assumption} and if $\tlambda \in \cB$ then $\lambda \in \cB$.
	\end{proof}
	
	\subsection{Decreasing the leftover parts of size $k$}

	After \cref{lem:wide-first-step}, the partition has most parts being of size $\lfloor k/2 \rfloor$.  We now simply use a small number of parts of size $\lfloor k/2 \rfloor$ to remove the $O(\sqrt{k})$ many parts of size larger than $\lfloor k/2 \rfloor$.
	
	\begin{lemma}\label{lem:leftover-k}
		Let $\lambda \vdash n$ satisfy \eqref{eq:part-ratio-assumption} and have at least $\frac{8}{9}A_0k - 90\sqrt{k}$ parts of size $\lfloor k/2 \rfloor$ and at most $100 \sqrt{k}$ parts of size $ k$.  Then there is a partition $\tlambda$ with $\tlambda_1 \leq \lfloor k/2 \rfloor$ with at least $\frac{8}{9}A_0k - 2^{20}\sqrt{k}$ parts of size $\lfloor k/2 \rfloor$ so that if $\tlambda \in \cB$ then $\lambda \in \cB$.
	\end{lemma}
	\begin{proof}
		The partition $\lambda$ has at most $101\sqrt{k}$ parts of size $> \lfloor k/2 \rfloor$.  
		Set $m := \lceil 101\sqrt{k} \rceil$.  We will make the first $m$ parts of the partition equal to size $m^2$ using some number of parts of size $\lfloor k/2 \rfloor$; let $\mu$ denote the resulting partition.  Notice that the total mass needed for this is at most $m^3 \leq 102^3 k^{3/2}$, and so the total number of parts of size $\lfloor k /2 \rfloor$ required to do this is at most $102^3 k / \lfloor k /2 \rfloor \leq (3\cdot 102^3) \sqrt{k}.$  
		
		By construction, note that $\la \preceq \mu$ and so by  \cref{lem:dominance-to-cont} we have that if $\mu \in \cB$ then $\lambda \in \cB$; since $3\cdot 102^3 + 90 < 2^{22}$, $\mu$ has at least $\frac{8}{9}A_0k - 2^{22}\sqrt{k}$ 
		parts of size $\lfloor k/2 \rfloor$; and that $\mu_{m+1} = \lfloor k/2\rfloor$.  Since the first $m$ parts of $\mu$ are $m^2$, by  \cref{lem:split} and \cref{fact:cube} if we let $\tlambda$ be $\mu$ with the first $m$ parts removed, then if $\tlambda \in \cB$ then $\mu \in \cB$.  Noting that $\tlambda$ satisfies $\tlambda_1 \leq \lfloor k/2 \rfloor$ and has at least $\frac{8}{9}A_0 k - 2^{20}\sqrt{k}$ parts of size $\lfloor k/2 \rfloor$ completes the Lemma. 
	\end{proof}
	
	\subsection{Proof of  \cref{prop:partition-wide}}
	
	We now combine the pieces from the previous subsections. 
	\begin{proof}[Proof of  \cref{prop:partition-wide}]
		We will prove the proposition by (strong) induction on $k$.  The base case of $k = 1$ follows from applying  \cref{lem:split} and \cref{fact:cube}.  
		
		To begin our inductive step, let $k \geq 2$.   Assume without loss of generality that each $\lambda,\mu,\nu$ has at most one part not of size $k$ and let $\lambda$ be this resulting partition.  By applying  \cref{lem:wide-first-step} and then \cref{lem:leftover-k}, we obtain a partition $\tlambda$ with $\tlambda_1 \leq \lfloor k/2 \rfloor$ which has at least $\frac{8}{9}A_0k - 2^{20}\sqrt{k}$ parts of size $\lfloor k/2 \rfloor$; we also have that if $\tlambda \in \cB$ then $\lambda \in \cB$. 
		
		By considering only the parts of size $\lfloor k/2 \rfloor$ we see that $$|\tlambda| \geq \left(\frac{8}{9}A_0 k - 2^{22} \sqrt{k}\right)  \lfloor k/2 \rfloor \geq   A_0\lfloor k/2\rfloor^2 $$
		provided $A_0 \geq 2^{24}$.
		We thus have that $\tlambda$ satisfies the hypotheses of the Proposition with the value of $\lfloor k / 2 \rfloor$ and so by induction we have $\tlambda \in \cB$, completing the proof.
	\end{proof}

	\section{Handling partitions with no large parts} \label{sec:no-large}
	In this section we show how to handle partitions whose largest part is on the order of $\sqrt{n}$. These partitions will contain many parts of this size so we will take away these parts in several steps until the resulting partition satisfies  \cref{prop:partition-wide}.
	\begin{proposition}\label{prop:meat}
		For each $\theta \in (0,1)$ if $A_1 \geq 3 \theta^{-1}A_0$ then the following holds.  For each $B > 0$, suppose  $(\lambda,\mu,\nu)$ are partitions of $n$ where each has largest part at most $B \sqrt{n}$ and for $\rho \in \{\lambda,\mu,\nu\}$ we have \begin{equation}\label{eq:tail-mass-assumption}
			\sum_{j} \rho_{j} \one_{\rho_j \leq \sqrt{n}/A_1} \geq \theta\frac{n}{A_1}\,.
		\end{equation}
		Then for $n$ sufficiently large (as a function of $\theta,B,A_1$) we have  $(\lambda,\mu,\nu) \in \cB$.
	\end{proposition}

	The strategy will be similar to the one employed in  \cref{sec:wide}: we will iteratively remove the larger pieces until all parts are of size at most $\sqrt{n}/A_1$ with size at least $\theta n/(2 A_1)$.  If $A_1$ is large enough as a function of $\theta$ and $A_0$ then  we may apply  \cref{prop:partition-wide} to handle this remaining triple.  As in  \cref{sec:wide}, we will use  \cref{lem:dominance-to-cont} in order to assume that all partitions are equal at each state of the proof of  \cref{prop:meat}.  
	
	Throughout this section, we write $a = \lfloor \sqrt{n}/A_1 \rfloor$ and $b = \lfloor B\sqrt{n}\rfloor$.

	\subsection{Decreasing most parts of size $b$}

	In analogy with  \cref{sec:wide}, it will be useful to maintain that all but a small number of parts are of size $a$ and $b$.  As such, we say a partition satisfies \eqref{eq:part-ab} if it has \begin{align}
		\text{at most one part with size in } (a,b) \nonumber \\
		\text{ at most one part with size } < a, \label{eq:part-ab} \tag{$**$}  \\ 
		\text{ all other parts of size }a \text{ or }b\,.  \nonumber 
	\end{align}
	
	We first remove most parts of size $b$. 
	
	\begin{lemma}\label{lem:meat-most}
		Let $\lambda \vdash n$  satisfy \eqref{eq:part-ab} and have at least $\sqrt{n}/2$ parts of size $a$.  Then there is a partition $\tlambda$ satisfying \eqref{eq:part-ab} so that $\tlambda$ has at most $3\sqrt{b}$ many parts of size $b$ and so that if $\tlambda \in \cB$ then $\lambda \in \cB$.
	\end{lemma}
	
	We define an operator $T_1$ on partitions satisfying \eqref{eq:part-ab} which will remove parts of size $b$.  Take $m$ to be the smallest integer such that $m^2 \geq b$ and note that \begin{equation}\label{eq:b-m-bounds}
		b \leq m^2 \leq (\sqrt{b}+1)^2 \leq b + 2\sqrt{b} + 1\,.
	\end{equation} 
	
	We will make the first $m$ parts of $\lambda$ to be of size $m^2$ by reducing parts of size $b$ to size $0$.  After this, if there is more than one part with size in  $(a,b)$ or more than one with size $< a$ then move pieces from the smaller part to the larger part until one has either size $a$ or $b$.  These operations increase $\lambda$ in the dominance order by  \cref{lem:prec}.  
	Once the first $m$ parts are of size $m^2$, we  remove the first $m$ parts to obtain $\tlambda$.  
	
	We first record the domain of $T_1$.  
	
	\begin{lemma}\label{lem:T_1-domain}
		The operator $T_1$ may be performed on a partition of $\lambda$ satisfying \eqref{eq:part-ab} provided $\lambda$ has at least $3\sqrt{b}$ parts of size $b$.
	\end{lemma}
	\begin{proof}
		We require increasing the first $m$ marginals from $b$ to $m$, which requires a total mass of $$m(m - b) \leq (\sqrt{b} +1) (2\sqrt{b} + 1) \leq 3 b $$ for $n$ large enough; this requires at least $3$ parts of size $b$.  We also require $m \leq 2\sqrt{b}$, thus completing the Lemma.
	\end{proof}

	We record the two main properties of $T_1$. 
	\begin{lemma}\label{lem:T_1-props}
		The operation $T_1$ takes in a partition $\lambda$ satisfying \eqref{eq:part-ab} and outputs a partition $T_1(\lambda) = \tlambda$ satisfying \eqref{eq:part-ab} and if $\tlambda \in \cB$ then $\lambda \in \cB$.
	\end{lemma}
	\begin{proof}
		By construction, only the parts of size $b$ are altered and we yield a partition with at most one part in each interval $(a,b)$ and $(0,a)$.  At each step we have either increased the partition in the dominance order or removed a partition in $\cB$, so the conclusion follows from  \cref{lem:split} and \cref{fact:cube}.
	\end{proof}

	Iterating $T_1$ yields  \cref{lem:meat-most}:
	
	\begin{proof}[Proof of  \cref{lem:meat-most}]
		Apply $T_1$ to $\lambda$ until the resulting partition is outside of its domain.  By  \cref{lem:T_1-props} we yield a partition $\tlambda$ satisfying \eqref{eq:part-ab} so that if $\tlambda \in \cB$ then $\lambda \in \cB$.  Finally, since $T_1$ cannot be applied to $\tlambda$,  \cref{lem:T_1-domain} shows that $\tlambda$ has less than $3\sqrt{b}$ parts.
	\end{proof}
	
	\subsection{Decreasing the leftover parts of size $b$ and proving  \cref{prop:meat}}
	
	\begin{lemma} \label{lem:meat-leftover}
		Let $\lambda \vdash n$ satisfy \eqref{eq:part-ab} and have at least $\theta \sqrt{n}$ parts of size $a$ and less than $3\sqrt{b}$ parts of size $b$.  If $A_1 \geq 3\theta^{-1} A_0^2$, then $\lambda \in \cB$.
	\end{lemma}
	\begin{proof}
		There are at most $4\sqrt{b}$ parts of size greater than $  a$.  Set $M = \lceil 4 \sqrt{b} \rceil$.  We will use parts of size $a$ to increase the first $M$ parts to size $M^2$.  This increases the partition in the dominance order, and so if we remove the first $M$ parts after this increase then we obtain a partition $\tlambda$ so that if $\tlambda \in \cB$ then $\lambda \in \cB$ by  \cref{lem:dominance-to-cont}, \cref{lem:split} and \cref{fact:cube}.
		
		Increasing the first $M$ parts to size $M^2$ requires a total mass of at most $$M^3 \leq 65 B^{3/2} n^{3/4} $$ which can be provided by $(66 B^{3/2} A_1)n^{1/4}$ parts of size $a$; we may additionally require at most  $M \leq 5 \sqrt{B} n^{1/4}$ parts of size $a$, as up to $M - 1$ many may be increased.  This shows $\tlambda$ has at least $\theta\sqrt{n} - O_{A_1,B}(n^{1/4}) \geq \theta \sqrt{n}/2$ many parts of size $a$ for $n$ large enough.   This shows that $$|\tlambda| \geq \frac{\theta\sqrt{n}}{2}\cdot a = (1 + o(1)) \frac{\theta n}{2A_1}\,.$$
		For $A_1 \geq 3 \theta^{-1} A_0^2$ we thus have $$a = \lfloor \sqrt{n} / A_1 \rfloor \leq \frac{\sqrt{|\tlambda|} }{A_0}$$
		for $n$ large enough as a function of $\theta$ and $B$
		and so we may apply  \cref{prop:partition-wide} to see $\tlambda \in \cB.$  This shows $\lambda \in \cB$.
	\end{proof}
	
	\begin{proof}[Proof of  \cref{prop:meat}]
		For each integer $n$ construct a partition $\rho$ by setting there to be $\sqrt{n}/2$ parts of size $a$, create as many remaining parts of size $b$ and have one leftover part (which may be of any size $< b$).  Note that this partition $\rho$ satisfies \eqref{eq:part-ab} and that $\lambda \preceq \rho, \mu \preceq \rho$ and $\nu \preceq \rho$ by \eqref{eq:tail-mass-assumption}.  By  \cref{lem:dominance-to-cont} we may assume without loss of generality that $\lambda = \mu = \nu = \rho$.  Applying  \cref{lem:meat-most} and \cref{lem:meat-leftover} completes the proof. 
	\end{proof}

	\section{The final step: removing large parts} \label{sec:large-parts}
	
	The goal of this section is to remove the large parts of the three partitions. Each partition may have an exceptional large part so we show that we can fill out the table using that part so that the remaining partition satisfies the conditions of  \cref{prop:meat}.

	\begin{proof}[Proof of  \cref{th:deterministic}]
		Set $a = \lfloor \sqrt{n} / A_0\rfloor$ for $A_0$ to be taken sufficiently large in the proof.  We can only increase each partition in the dominance order by assuming that the first part of each partition is of size $a^2 \geq n /(2A_0^2)$ and all other parts are of size at most $B\sqrt{n}$.  By  \cref{lem:dominance-to-cont} we may assume this holds for each partition without loss of generality.  Further, we may assume that each partition has at least $\theta \sqrt{n}$ parts of size $a$.  For $A_0$ large enough we have $1/A_0 \leq \theta/3$, and so we will assume this holds.

		We will prove the theorem in three steps, where we first remove the large part of $\lambda$ then $\mu$ then $\nu$.  Let $S_2$ and $S_3$ be intervals of size $a$ for which we have $\mu_j = a$ for all $j \in S_2$ and $\nu_k = a$ for all $k \in S_3$.
		
		To remove the large part of $\lambda$, note that if we consider the partial contingency table $M$ where we set $M(1,j,k) = 1$ for $j \in S_2, k\in S_3$, then this fixes the marginals corresponding to $\lambda_1, \mu_j$ for $j \in S_2$ and $\nu_k$ for $k \in S_3$.  Applying  \cref{lem:split} shows that removing $\lambda_1$, $\mu_j$ for $j \in S_2$ and $\nu_k$ for $k \in S_3$ yields partitions $\tlambda,\tmu,\tnu$ for which $(\tlambda,\tmu,\tnu) \in \cB$ implies $(\lambda,\mu,\nu) \in \cB$.  Repeating this to remove the large part of $\mu$ and then the large part of $\nu$ yields partitions $(\tlambda,\tmu,\tnu)$ with largest part at most $B\sqrt{n}$ so that if $(\tlambda,\tmu,\tnu) \in \cB$ then $(\lambda,\mu,\nu) \in \cB$.  The only thing left to check is that each subsequent partition satisfies the hypotheses of  \cref{prop:meat}.  Note that each has largest part $B \sqrt{n}$ and that they each have at least $\theta \sqrt{n} /3$ parts of size $a$.  Finally, note that $|\tlambda| = n 
		- 2a^2 = n(1 - O(\frac{1}{A_0^2}))\,.$
		For $A_0$ sufficiently large, we thus have $$\sum_{j} \tlambda_j \one_{\tlambda_j \leq 2 \sqrt{|\tlambda|}/A_0} \geq \sum_{j} \tlambda_j \one_{\tlambda_j \leq a} \geq \frac{\theta \sqrt{n} a}{3} \geq \frac{\theta |\tlambda| }{8 (A_0 / 2)\,.}  $$
		Thus, taking $A_0$ large enough as a function of $\theta$ allows us to apply  \cref{prop:meat}, completing the proof.
	\end{proof}

	\section{A shape theorem and the proof of  \cref{th:main}} \label{sec:shape-thm}
	Classical theorems of Vershik \cite{vershik1996statistical} and Fristedt \cite{fristedt1993structure} prove a \emph{shape theorem} for random partitions; namely, if one rescales the Ferrer's diagram of a random partition by $\sqrt{n}$ in each direction, then the diagram adheres closely to a deterministic shape.  A sufficiently strong version of the shape theorem will show that a random partition satisfies \eqref{eq:shape-assumptions} with probability at least $1 - e^{-c\sqrt{n}}$. We refer the reader to the surveys of Shlosman \cite{shlosman2001wulff} and Okounkov \cite{okounkov2016limit} for more on the limit shapes of partitions.  Throughout this section, we write $\P_n$ for the probability measure on partitions $\lambda \vdash n$ chosen uniformly at random.
	
	\begin{theorem}\label{th:shape}
		For each $[t_0,t_1] \subset (0,\infty)$ and $\eps > 0$ there is a $c = c(\eps,t_0,t_1) > 0$ so that 
		$$ \P_n\left( \max_{[s_0,s_1] \subset [t_0,t_1]} \left|\frac{1}{n}\sum_{j} \lambda_j \one_{\lambda_j/ \sqrt{n} \in [s_0,s_1]} -  \int_{s_0}^{s_1} \frac{x}{e^{\frac{\pi}{\sqrt{6}}x} - 1} \,dx\right| \geq \eps \right) \leq \exp(-c \sqrt{n})\,.$$
	\end{theorem}
	This version of the shape theorem follows from, e.g., \cite[Corollary 5.2]{MMP}. We will deduce the following corollary from  \cref{th:shape}.
	
	\begin{corollary}\label{cor:limit-shape}
		For each $\theta \in (0,\sqrt{6}/\pi)$ there is a $A_0 > 0$ so that for each $A \geq A_0$ there is a constant $c  = c(\theta,A) > 0$ so that $$\P_n\left(\lambda \text{ satisfies }\eqref{eq:shape-assumptions} \text{ with parameters }(\theta,A,A) \right) \geq 1 - 2e^{-cn}\,.$$ 
	\end{corollary}
	\begin{proof}
		Note first that for $T$ sufficiently large we have $$\int_T^\infty \frac{x}{e^{\frac{\pi}{\sqrt{6}}x} - 1}\,dx \leq e^{-T}$$ since $\pi/\sqrt{6} > 1$.  Additionally, by the fundamental theorem of calculus we have that $$\lim_{t\to 0} t^{-1} \int_0^t \frac{x}{e^{\frac{\pi}{\sqrt{6}}x} - 1} \,dx =\frac{\sqrt{6}}{\pi}\,. $$
		In particular, for $\theta' \in (\theta,\sqrt{6}/\pi)$ we can take $A_0$ sufficiently large so that for all $A \geq A_0$ we have $$\int_{1/A^2}^{1/A} \frac{x}{e^{\frac{\pi}{\sqrt{6}}x} - 1}\,dx \geq \frac{\theta'}{A}\qquad \text{ and } \qquad \int_{1/A^3}^{A} \frac{x}{e^{\frac{\pi}{\sqrt{6}}x} - 1}\,dx \geq 1 - \frac{1}{4 A^2}\,.$$
		Applying  \cref{th:shape} with $[t_0,t_1] = [A^{-3},A]$ and taking $\eps \leq \min\{\frac{\theta - \theta'}{A},\frac{1}{4A^2}\}$ shows \begin{align*}
			\P_n\left(\frac{1}{n} \sum_{j} \lambda_j \one_{\lambda_j/\sqrt{n} \in [A^{-3},A]} \leq 1 - \frac{1}{2A^2} \right) \leq e^{-c\sqrt{n}}
		\end{align*}
		and 
		\begin{align*}
			\P_n\left(\frac{1}{n} \sum_{j} \lambda_j \one_{\lambda_j/\sqrt{n} \in [A^{-3},A]} \leq \frac{\theta}{A} \right) \leq e^{-c\sqrt{n}}\,.
		\end{align*}
		The complements of these two events imply the first and second properties in \eqref{eq:shape-assumptions} respectively. 
	\end{proof}
	
	\begin{proof}[Proof of  \cref{th:main}]
		Let $\theta = 1/2$ and $A_0$ as guaranteed by  \cref{th:deterministic}. Applying  \cref{cor:limit-shape} for $\theta = 1/2$ yields a constant $A_0'$.  Taking $A \geq \max\{A_0,A_0'\}$ completes the proof. 
	\end{proof}

	\section{Hypergraphs} \label{sec:hypergraphs}
	
	In this section we deduce  \cref{th:hypergraph} from  \cref{th:deterministic}.  Before doing so, we recall a few relevant definitions.
	
	A three-uniform hypergraph $H=(V,E)$ consists of a set of vertices $V$ and an edge set $E$ where each edge is made up of exactly three distinct vertices; in particular if $e\in E$ then $e=\{u,v,w\}$ for some $u,v,w\in V$.  The \emph{degree} of a vertex $v \in V$ is the number of edges containing $e$ and is denoted $\deg(v)$; when there is more than one relevant hypergraph, we will write $\deg_H(v)$ for the degree of $v$ in $H$.  For a partition $\lambda \vdash 3n$, write $\lambda \in \mathcal{H}_3$ if there is a 3-uniform hypergraph with degree sequence given by $\lambda$. 
	
	Before proving  \cref{th:hypergraph}, we first prove an analogue of  \cref{lem:dominance-to-cont}.
	
	\begin{lemma}\label{lem:hyper-dom}
		Let $\lambda,\mu \vdash 3n$ with $\lambda \preceq \mu$.  If $\mu \in \mathcal{H}_3$ then $\lambda \in \mathcal{H}_3$.
	\end{lemma}
	\begin{proof}
		Assume without loss of generality that $\lambda \precdot \mu$.  By  \cref{lem:prec}, there is an $i_0$ so that $$\lambda_{i_0} + 1 = \mu_{i_0}, \quad \lambda_{i_0 + 1} - 1 = \mu_{i_0 + 1}, \quad \text{ and } \lambda_i = \mu_i, \quad \text{ for all }i \notin \{i_0,i_0 + 1\}\,.$$
		Let $H$ be a hypergraph with degree sequence $\mu$, and let $v_1,v_2,\ldots$ be the vertices of $H$ ordered by degree.  Since $\mu_{i_0} > \mu_{i_0+1}$, there is a pair $k,\ell$ so that  is an edge $e:=\{v_{i_0},v_k,v_\l \} \in E(H)$ and $\widetilde{e} := \{v_{i_0+1},v_k,v_\l\} \notin E(H)$.  Define the hypergraph $\widetilde{H}$ by $V(\widetilde{H}) = V(H)$ and $E(\widetilde{H}) = E(H) - e + \widetilde{e}$.  
		
		We now claim that the degree sequence of $\widetilde{H}$ is $\lambda$.  Note first that $\deg_{\widetilde{H}}(v_i)=\la_i=\mu_i$ for $i \notin \{i_0, i_0 + 1\}$.  Further we have  
		$$\deg_{\widetilde{H}}(v_{i_0})=\deg_H(v_{i_0})-1=\mu_{i_0}-1=\la_{i_0}$$ and $$\deg_{\widetilde{H}}(v_{i_0+1})=\deg_H(v_{i_0+1})+1=\mu_{i_0+1}+1=\la_{i_0+1}. $$ Hence, $\widetilde{H}$ is the 3-uniform hypergraph realization of degree sequence $\la$.
	\end{proof}
	
	With our lemma on dominance orders now established, we now prove  \cref{th:hypergraph}.  The idea will be to break $\lambda \vdash 3n$ into three partitions of $n$ and then use  \cref{th:deterministic} to construct a tripartite hypergraph with degree sequence $\lambda$.  
	
	\begin{proof}[Proof of  \cref{th:hypergraph}]
		Let $\lambda \vdash 3n$ be chosen uniformly at random.  Applying  \cref{cor:limit-shape} with $\theta = 1/2$ we see that for each fixed $A$ large enough by  \cref{cor:limit-shape} we have \begin{align*}
			\sum_{j} \lambda_j \one_{\lambda_j > A\sqrt{3n}} &\leq \frac{3n}{2 A^2} \\
			\sum_{j} \lambda_j \one_{\lambda_j \leq \sqrt{3n}/A} &\geq \frac{1}{2} \cdot \frac{3n}{ A}
		\end{align*}
		with probability at least $1 - \exp(-\Omega(\sqrt{n}))\,.$
		
		We will increase $\lambda$ in the dominance order so that we may divide it into three partitions of $n$ while still satisfying some version of \eqref{eq:shape-assumptions}.  By  \cref{lem:hyper-dom} if we increase $\lambda$ in the dominance order to yield a partition in $\mathcal{H}_3$ then we will have that $\lambda \in \mathcal{H}_3$.
		
		First, we may combine all parts of $\lambda$ that are larger than $A \sqrt{3n}$ into a single part of mass $\leq 3n/2A^2$.  
		
		Second, we may take all parts of size $\leq \sqrt{3 n} / A$ and combine them into parts of size $a := \lfloor \sqrt{3 n} /A \rfloor$ with at most one part of size $< a$.  By the second displayed equation above, this will yield at least $\sqrt{n}$ parts of size $a$. 
		
		The last step will be to break up $\lambda$ into three partitions of $n$.  With this in mind, define $N := \lfloor \frac{\sqrt{n} }{4} \rfloor$.  The idea will be for each of the three partitions to have $N$ parts of size $a$.  In order to do so, first define $i_0$ to be the least $i$ so that $\sum_{j \leq i} \lambda_j \leq n - aN$.  For $A$ large enough, we have that $3n / (2 A^2) < n - aN$ and so we have $\lambda_2 \leq A\sqrt{3n}$; in particular, this implies that $n - aN\geq \sum_{j \leq i_0} \lambda_j \geq n - aN -A\sqrt{3n}$.  If $\lambda_1 - \lambda_2 > A\sqrt{3n}$ then we may decrease at most $2 A^2$ parts of size $a$ to increase $\lambda_2$ so that we exactly have \begin{align*}
			\sum_{j \leq i_0} \lambda_j = n - aN\,.
		\end{align*}
		Otherwise, we will increase $\lambda_1$ instead to yield the above displayed equation.  In either case, we have that e.g. \begin{align} \label{eq:hypergraph-large-parts}
			\sum_{j} \lambda_j \one_{\lambda_j > A\sqrt{4n}} \leq \frac{3n}{2 A^2} \end{align}
		for $n$ large enough.
		
		Now define $i_1$ to be the least $i$ so that $\sum_{i_0 <j \leq i} \lambda_j \leq n - aN$.  Again, we may decrease at most $2 A^2$ parts of size $a$ to increase $\lambda_{i_0 + 1}$ so that we exactly have $$\sum_{j = i_0 + 1}^{i_1} \lambda_{j} = n - aN\,.$$
		
		We note that by construction we have only used $O_A(1)$ many parts of size $a$ to increase the large parts.  
		
		We now define three partitions $\tlambda,\tmu,\tnu$: $\tlambda$ will consist of $(\lambda_1,\ldots,\lambda_{i_0})$ together with $N$ parts of size $a$; $\tmu$ will consist of $(\lambda_{i_0+1},\ldots,\lambda_{i_1})$ together with $N$ parts of size $a$; and $\tnu$ will consist of all remaining parts.  Note that by construction all three partitions are partitions of $n$ and that if $(\tlambda,\tmu,\tnu) \in \cB$ then $\lambda \in \mathcal{H}_3$.  
		
		Set $A' = A/\sqrt{3}$ and $B = 2A$.   We now note that for $\rho \in \{\tlambda,\tmu,\tnu\}$, equation \eqref{eq:hypergraph-large-parts} shows \begin{align*} 
			\sum_{j} \rho_j \one_{\rho_j > B\sqrt{n}} \leq \frac{n}{2 (A')^2}\,. \end{align*}
		Further, each has at least $N$ parts of size $a$, implying that $$\sum_j \rho_j \one_{\rho_j \geq \sqrt{n}/A'} \geq N\cdot a \geq \frac{n}{5 A'}\,. $$
		Allowing $A$ and $n$ to be sufficiently large, we then may apply  \cref{th:deterministic} with $\theta = 1/5$ to show $(\tlambda,\tmu,\tnu) \in \cB$, completing the proof.
	\end{proof}

	\subsection{Linear Hypergraphs} \label{sec:linear-hypergraph}
	A \textit{linear} hypergraph is a hypergraph where each pair of vertices lies in at most edge. One can ask a similar question whether a randomly chosen partition can be realized as the degree sequence of a linear 3-uniform hypergraph.  In contrast to  \cref{th:hypergraph}, we show that the probability of this realization tends to 0. 
	
	The main idea is to look at the shadow of the hypergraph.  The \textit{shadow} of a 3-uniform hypergraph is a 2-uniform hypergraph---i.e.\ an ordinary graph---on the same vertex set where each edge of the hypergraph is replaced with a triangle.  In particular, we can write the shadow graph $\partial H$ to have the same vertex set of $H$ and an edge between $u$ and $v$ if there is some $w$ so that $\{u,v,w\} \in E(H)$.  Crucially, if $H$ is a linear hypergraph, then the shadow graph $\partial H$ is a simple graph, meaning no edge has multiplicity greater than $2$.

	Once we view our 3-uniform linear hypergraph through its shadow, we can appeal to results on graph realizability to bound the probability such a linear hypergraph exists.  The idea is that the Erd\H{o}s-Gallai condition on the shadow graph will yield a system of inequalities on $\lambda$ for which the proof of \cite{MMM} directly applies.
	
	\begin{proposition}
		Let $\la\vdash 3n$ be chosen uniformly at random. Then the probability there is a $3$-uniform linear hypergraph with degree sequence $\lambda$ is at most $n^{-.003}$.  
	\end{proposition}
	\begin{proof}
		Suppose that $\la$ was indeed the degree sequence for some linear hypergraph $H$. Then we can set $G = \partial H$ to be its shadow.  If we set $\mu$ to be the degree sequence of $G$, then we have $\mu_j = 2 \lambda_j$ for all $j$.  Since $G$ is a graph, by the Erd\H{o}s-Gallai criterion that 
		
		$$\sum_{i=1}^k \mu_i' \geq \sum_{i=1}^k \mu_i +i \quad \text{ for all }\quad k\leq D(\mu)$$ where $D(\mu)$ represents the size of the largest Durfee square in the Ferrers diagram of $\mu$.   While the approach of Pittel \cite{PITTEL1999123} will give a probability that this tends to $0$, we will follow \cite{MMM} instead for a stronger bound.  
		
		By a quantitative shape theorem of Pittel \cite{pittel2018} (see  \cite[Lemma 3]{MMM} for an explicit version), we have that $D(\mu) \geq n^{-\gamma}$ with probability at least $1 - \Omega(n^{-1/2 + 2\gamma} \log n)$ where $\gamma \approx .25$ (see the proof of \cite[Theorem 2]{MMM} for an optimal choice of $\gamma$ for this proof).
		
		In particular, we then have that \begin{align*}
			\P(\lambda &\text{ is the degree sequence of a linear }3\text{-uniform hypergraph}) \\
			&\quad\leq \P\left(\sum_{i = 1}^k \mu_i' \geq \sum_{i = 1}^k (\mu_i + i) \text{ for all } k \leq D(\mu) \right) \\
			&\quad\leq O(n^{-1/2 + 2\gamma} \log n) + \P\left(\sum_{i = 1}^k \lambda_i' \geq \sum_{i = 1}^k (\lambda_i + i/2) \text{ for all } k \leq n^\gamma \right)\,.
		\end{align*}
		
		By the proof of \cite[Theorem 2]{MMM} this probability is at most $n^{-.003}$.
	\end{proof}
	
	\section{Pyramid Shapes} \label{sec:pyramid}
	
	Recall that for three partitions $\lambda,\mu,\nu \vdash n$, $\mathrm{Pyr}(\lambda,\mu,\nu)$ is the number of pyramids with marginals $(\lambda,\mu,\nu)$.  Here we prove \cref{prop:pyramid-limit} and confirm the conjecture of Pak and Panova to show that $\mathrm{Pyr}(\lambda,\mu,\nu)$ is typically $0$.
	
	The idea will be to identify some inequality amongst the marginals $(\lambda,\mu,\nu)$ of a pyramid and then use the theorem of Pittel or Melczer-Michelen-Mukherjee to upper bound the probability.
	
	One way to view a pyramid shape is as a set of columns on a Ferrer's diagram; if we record the height of each column and place it in the diagram then we obtain a plane partition.  We first observe that the height of these columns---or equivalently the numbers in the plane partition---precisely gives the transpose of one of the marginals:
	\begin{lemma}\label{lem:columns}
		Let $M$ be a pyramid with marginals given by $(\lambda,\mu,\nu)$.   Let $\rho$ be the partition whose parts are equal to $(\sum_{k} M(i,j,k))_{i,j}$.  Then $\rho = \nu'$.
	\end{lemma}
	\begin{proof}
		For each $r$ compute \begin{align*}
			\rho_r' = \left| \left\{ (i,j) : \sum_{k} M(i,j,k) \geq r  \right\}\right| = \left| \left\{ (i,j) :  M(i,j,r) = 1  \right\}\right| = \sum_{i,j} M(i,j,r)  = \nu_r
		\end{align*}
		where the second equality is due to the fact that $M$ is a pyramid.
	\end{proof}
	
	From here we then observe an ordering that will quickly provide  \cref{prop:pyramid-limit}:
	
	\begin{lemma}\label{lem:pyramid}
		Let $M$ be a pyramid with $|M| = n$ and marginals $(\lambda,\mu,\nu)$.  Then $$\nu' \preccurlyeq \lambda\,.$$ 
	\end{lemma}
	\begin{proof}
		By  \cref{lem:columns}, note that we may find distinct pairs $(i_1,j_1),(i_2,j_2),\ldots $ so that for each $r$ we have $\nu_r' = \sum_k M(i_r,j_r,k)$.  Further, since $M$ is a pyramid we may choose $(i_r,j_r)$ so that the selection of indices $\{(i_\ell,j_\ell) : \ell \leq r\} $ forms a Ferrer's diagram for each $r$.  In this case, we immediately have that $i_r,j_r \leq r$ for all $r$.
		
		For each $r$, we thus have \begin{equation*}
			\lambda_1 + \cdots +\lambda_r = \sum_{i \leq r} \sum_{j,k} M(i,j,k) \geq \sum_{i,j \leq r} \sum_{k} M(i,j,k) \geq \sum_{\ell = 1}^r \sum_{k} M(i_\ell,j_\ell,k) = \nu_1' + \cdots + \nu_k'\,. \qedhere \end{equation*}
	\end{proof}
	
	\begin{proof}[Proof of  \cref{prop:pyramid-limit}]
		By \cref{lem:pyramid} we have $$\P_n(\mathrm{Pyr}(\lambda,\mu,\nu) \geq 1) \leq \P_n(\nu' \preccurlyeq \lambda) = O(n^{-.003})$$
		where the bound is by \cite{MMM}.
	\end{proof}

	\section*{Acknowledgments}
	\noindent Both authors are supported in part by NSF grants DMS-2137623 and DMS-2246624. The authors thank Matthew Jenssen for comments on a previous draft.

	\bibliographystyle{abbrv}
	\bibliography{refs}
	
\end{document}